\documentclass[12pt,a4paper]{amsart}
\usepackage{latexsym,amsfonts,amsthm,amsmath,mathrsfs,amssymb}
\usepackage[british]{babel}
\usepackage[dvips]{color}
\usepackage[utf8]{inputenc}
\usepackage[T1]{fontenc}
\usepackage[dvips,final]{graphics}
\usepackage{xcolor}
\usepackage{mathrsfs}
\usepackage{mathtools} 
\usepackage{breqn}
\usepackage{epstopdf}
\usepackage{verbatim}
\usepackage{amscd}
\usepackage{hyperref}
\usepackage[english,capitalize]{cleveref}
\usepackage{aliascnt}

\usepackage{todonotes}

\makeatletter
\def\newaliasedtheorem#1[#2]#3{
	\newaliascnt{#1@alt}{#2}
	\newtheorem{#1}[#1@alt]{#3}
	\expandafter\newcommand\csname #1@altname\endcsname{#3}
}
\makeatother

\theoremstyle{plain}

\newtheorem{theorem}{Theorem}[section]
\newaliasedtheorem{prop}[theorem]{Proposition}
\newaliasedtheorem{lem}[theorem]{Lemma}
\newaliasedtheorem{coroll}[theorem]{Corollary}

\theoremstyle{definition}
\newaliasedtheorem{defi}[theorem]{Definition}
\newaliasedtheorem{quest}[theorem]{Question}
\newaliasedtheorem{fact}[theorem]{Fact}

\theoremstyle{remark}
\newaliasedtheorem{rem}[theorem]{Remark}
\newaliasedtheorem{exa}[theorem]{Example}

\newcommand{\R}{\mathbb{R}}

\newcommand{\C}{\mathbb{C}}

\let\altphi\phi
\let\phi\varphi
\let\varphi\altphi
\let\altphi\undefined

\newcommand{\average}{{\mathchoice {\kern1ex\vcenter{\hrule height.4pt
width 6pt
depth0pt} \kern-9.7pt} {\kern1ex\vcenter{\hrule height.4pt width 4.3pt
depth0pt}
\kern-7pt} {} {} }}

\address{\textsc{Daniela Di Donato}: 
Dipartimento di Ingegneria Industriale e Scienze Matematiche, Via Brecce Bianche, 12 60131 Ancona, Universit\'a Politecnica delle Marche.}
\email{daniela.didonato@unitn.it}
\title{Ahlfors-David regularity of intrinsically quasi-symmetric sections in metric spaces}

\date{\today}

\author{ Daniela Di Donato }

\subjclass[]{ 
26A16  
51F30 
46B04 
54E35
}
\keywords{Quasi-symmetric graphs, Ahlfors-David regularity, Metric spaces}

\begin{document}

\begin{abstract}
		We introduce a definition of  intrinsically quasi-symmetric sections in metric spaces and we  prove the Ahlfors-David regularity for this class of sections. We follow a recent result by Le Donne and the author where we generalize the notion of intrinsically Lipschitz graphs in the sense of Franchi, Serapioni and Serra Cassano.  We do this by focusing our attention on the graph property instead of the map one.	\end{abstract}

\maketitle

\section{Introduction}
The notion of Lipschitz maps is a key one for rectifiability theory in metric spaces \cite{Federer}. On the other hand, in \cite{AmbrosioKirchheimRect} Ambrosio and Kirchheim prove that the classical Lipschitz definition not work in the context of SubRiemannian Carnot groups \cite{ABB, BLU, CDPT}. Then in a similar way of Euclidean case, Franchi, Serapioni and Serra Cassano \cite{FSSC, FSSC03, MR2032504}  introduce a suitable definition of intrinsic cones which is deep different to Euclidean cones and then they say that a map $\phi$ is intrinsic Lipschitz if  for any $ p\in  \mbox{graph}(\phi) $ it is possible to consider an intrinsic cone $\mathcal{C}$ with vertex on $p$ such that
\begin{equation*}
\mathcal{C} \cap \mbox{graph}(\phi) = \emptyset.
\end{equation*}
In  \cite{DDLD21}, we generalize this concept in general metric spaces. Roughly speaking, in our new approach a section $\psi$ is such that $\mbox{graph}(\phi) =\psi (Y) \subset X$ where $X$ is a metric space and $Y$ is a topological space. We prove some important properties as the Ahlfors regularity, the Ascoli-Arzel\'a Theorem, the Extension theorem for so-called intrinsically Lipschitz sections. Following this idea, the author introduce other two natural definitions: intrinsically H\"older  sections \cite{D22.1} and intrinsically quasi-isometric sections  \cite{D22.2} in metric spaces. Yet, thanks to the seminal paper \cite{C99} (see also \cite{K04, KM16}) it is possible to found suitable sets of this class of sections in order to get the convexity and vector space over $\R$ and $\C.$

Following Pansu in \cite{P21}, the purpose of this note is to give a natural   intrinsically quasi-symmetric notion and then, following again  \cite{DDLD21}, to prove the Ahlfors-David regularity result for this class of sections which includes intrinsically Lipschitz sections. More precisely, the main result of this paper is Theorem \ref{thm2}.

\subsection{Quasi-symmetric sections}
 Before to give a suitable definition of quasi-symmetric sections, we recall the classical notion of quasi-conformal maps \cite{10.1007/BF02392747, A06, H01, IM01, TV}.  Let $X$ and $Y$ be two metric spaces and let $f:Y\to X$ be an homeomorphism (i.e., $f$ and its inverse are continuous maps). For $\bar y \in Y, r>0$ we define
 \begin{equation*}
\begin{aligned}
L_f (\bar y, r) &:= \sup  \{ d(f(\bar y), f(y))\, :\, d(\bar y, y) \leq r \},\\
\ell_f (\bar y, r) &:= \inf  \{ d(f(\bar y), f(y))\, :\, d(\bar y, y) \geq r \},\\
\end{aligned}
\end{equation*}
and the ratio $H_f (\bar y , r ) := L_f(\bar y , r) / \ell _f(\bar y , r)$ which measures the eccentricity of the image of the ball $B(\bar y , r)$ under $f.$ We say that $f$ is $H$-quasiconformal if
\begin{equation}\label{equationHF}
\limsup _{r\to 0} H_f (\bar y, r) \leq H, \quad \forall \bar y \in Y.
\end{equation}
 A good point of our research is that $Y$ is just a topological space because, in many cases, we just consider the metric on $X.$ On the other hand, we can not do a automatically choice of $\ell _f$ and the reason will be clear after to present our setting. We have a metric space $X$, a topological space $Y$, and a 
quotient map $\pi:X\to Y$, meaning
continuous, open, and surjective.
The standard example for us is when $X$ is a metric Lie group $G$ (meaning that the Lie group $G$ is equipped with a left-invariant distance that induces the manifold topology), for example a subRiemannian Carnot group, 
and $Y$ is the space of left cosets $G/H$, where 
$H<G$ is a  closed subgroup and $\pi:G\to G/H$ is the projection modulo $H$, $g\mapsto gH$.

 In \cite{DDLD21}, we consider a section $\phi :Y \to X$ of  $\pi:X \to Y$ (i.e., $\pi \circ \phi = id_Y$) such that $\pi$ produces a foliation for $X,$ i.e., $X= \coprod \pi ^{-1} (y)$ and the Lipschitz property of $\phi$ consists to ask that the distance between two points $\phi (y_1), \phi (y_2)$ is not comparable with the distance between $y_1$ and $y_2$ but between $\phi (y_1)$ and the fiber of $y_2$.  Following this idea, the corresponding notion given in \eqref{equationHF} becomes
\begin{equation}
\limsup _{r\to 0} H_ \phi (\bar y, r) \leq H, \quad \forall \bar y \in Y,
\end{equation}
where  \begin{equation*}
\begin{aligned}
L_\phi (\bar y, r) &:= \sup  \{ d( \phi (\bar y), \phi(y))\, :\, d(\phi (\bar y), \pi ^{-1} (y)) \leq r \},\\
\ell_\phi (\bar y, r) &:= \inf  \{ d(\phi (\bar y), \phi (y))\, :\,  d(\phi (\bar y), \pi ^{-1} (y)) \geq r \},\\
\end{aligned}
\end{equation*}
and the intrinsic ratio $H_f (\bar y , r ) := L_\phi(\bar y , r) / \ell_\phi(\bar y , r).$ 

Now we can understand why we can not choice $\ell_\phi.$ Indeed, in this case, $$r\leq d(\phi (y_1), \pi ^{-1} (y_2)) \leq d(\phi (y_1), \phi (y_2))$$ and so 
  \begin{equation*}
\begin{aligned}
\ell_f (\bar y, r) =r, \quad \forall \bar y \in Y.\\
\end{aligned}
\end{equation*}

Because of this, we follow Pansu in \cite{P21}, and we give the following non-trivial definition.

\begin{defi}\label{def_ILS}We say that a map $\phi:Y\to X$ is an {\em intrinsically $\eta$-quasi-symmetric section of $\pi$}, if it is a section, i.e.,
\begin{equation}
\pi \circ \phi =\mbox{id}_Y,
\end{equation}
and if there exists an homeomorphism $\eta :(0,\infty) \to (0,\infty)$ (i.e., $\eta$ and its inverse are continuous maps) measuring the intrinsic quasi-symmetry of $\phi .$ This means that  for any $y_1, y_2,y_3\in Y$ distinct points of $Y$ which not belong to the same fiber, it holds
\begin{equation}\label{equation10aprile}
\frac{d(\phi (y_1), \phi (y_2))}{d(\phi (y_1), \phi (y_3))}\leq \eta\left( \frac{d(\phi (y_2), \pi^{-1} (y_1))}{d(\phi (y_3), \pi^{-1} (y_1))} \right).
\end{equation}
Here $d$ denotes the distance on $X$, and, as usual, for a subset $A\subset X$ and a point $x\in X$, we have
$d(x,A):=\inf\{d(x,a):a\in A\}$.
\end{defi}

Equivalently to \eqref{equation10aprile}, we are requesting that 
\begin{equation}
\frac{d(x_1, x_2)}{d(x_1, x_3)}\leq \eta\left( \frac{d( x_2, \pi^{-1} (\pi (x_1)))}{d( x_3, \pi^{-1} (\pi(x_1)))} \right), \quad \mbox{for all } x_1,x_2,x_3 \in \phi (Y),
\end{equation}
where we ask that $x_2, x_3 \notin \pi^{-1} (\pi(x_1)).$

 We give some examples of this class of maps.
\begin{exa}[Intrinsically Lipschitz section of $\pi$]\label{intrinsicLipschitz1} Following \cite{DDLD21}, we say that a map $\phi:Y\to X$ is an {\em intrinsically Lipschitz section of $\pi$ with constant $L$},  with $L\in[1,\infty)$, if it is a section and 
\begin{equation*}
d(\phi (y_1), \phi (y_2)) \leq L d(\phi (y_1), \pi ^{-1} (y_2)), \quad \mbox{for all } y_1, y_2 \in Y.
\end{equation*}
Here, $\eta (x) = Lx$ for every $x\in (0,\infty)$. Indeed,
     \begin{equation*}
\begin{aligned}
\frac{d(\phi (y_1), \phi (y_2))}{d(\phi (y_1), \phi (y_3))} & =  \frac{d(\phi (y_1), \phi (y_2))}{d(\phi (y_2), \pi^{-1} (y_1))} \frac{d(\phi (y_3), \pi^{-1} (y_1))}{d(\phi (y_1), \phi (y_3))} \frac{d(\phi (y_2), \pi^{-1} (y_1))}{d(\phi (y_3), \pi^{-1} (y_1))}\\
& \leq L \frac{d(\phi (y_2), \pi^{-1} (y_1))}{d(\phi (y_3), \pi^{-1} (y_1))},
\end{aligned}
\end{equation*}
where in the last inequality we used the simply fact $\phi (y_1) \in \pi^{-1} (y_1)$ and so $\frac{d(\phi (y_3), \pi^{-1} (y_1))}{d(\phi (y_1), \phi (y_3))} \leq 1.$
\end{exa} 

\begin{exa}[BiLipschitz embedding]  BiLipschitz embedding are examples of intrinsically $\eta$-quasi-symmetric sections of $\pi.$ This follows because 
in the case  $\pi$ is a Lipschitz quotient or submetry \cite{MR1736929, Berestovski}, being intrinsically Lipschitz  is equivalent to biLipschitz embedding, (see Proposition 2.4 in \cite{DDLD21}).
\end{exa}  

\begin{exa}[Intrinsically H\"older section of $\pi$ (in the discrete case)]\label{intrinsicHolder2} 
Let $X$ be a metric space with the additional hypothesis that there is $\varepsilon >0$ such that $d(\phi (y_1), \phi (y_2)) \geq \varepsilon >0$ for any $y_1, y_2 \in Y.$ 
 Following \cite{D22.1}, we say that a map $\phi:Y\to X$ is an {\em intrinsically $(L, \alpha)$-H\"older section of $\pi$},  with $L\in[1,\infty)$ and $\alpha \in (0,1)$, if it is a section and 
\begin{equation*}
d(\phi (y_1), \phi (y_2)) \leq L d(\phi (y_1), \pi ^{-1} (y_2))^\alpha, \quad \mbox{for all } y_1, y_2 \in Y.
\end{equation*}
Here, $\eta (x) = L \varepsilon ^{\alpha -1}x^\alpha $ for any $x\in (0,\infty)$. Indeed,
  \begin{equation*}
\begin{aligned}
\frac{d(\phi (y_1), \phi (y_2))}{d(\phi (y_1), \phi (y_3))} & =  \frac{d(\phi (y_1), \phi (y_2))}{d(\phi (y_2), \pi^{-1} (y_1)) ^\alpha } \frac{d(\phi (y_3), \pi^{-1} (y_1))^\alpha}{d(\phi (y_1), \phi (y_3))} \frac{d(\phi (y_2), \pi^{-1} (y_1))^\alpha}{d(\phi (y_3), \pi^{-1} (y_1))^\alpha}\\
& \leq L  \varepsilon^{\alpha -1} \frac{d(\phi (y_2), \pi^{-1} (y_1)) ^\alpha}{d(\phi (y_3), \pi^{-1} (y_1))^\alpha},
\end{aligned}
\end{equation*}
as desired.
\end{exa}

%
%
%

\section{Ahlfors-David regularity}\label{theoremAhlforsNEW} 
Regarding Ahlfors-David  regularity in metric setting, the reader can see  \cite{DDLD21} for intrinsically Lipschitz sections;  \cite{D22.1} for H\"older sections; \cite{D22.2} for intrinsically quasi-isometric sections.

The main result of this paper is the following.
   \begin{theorem}[Ahlfors-David regularity]\label{thm2}
   Let $\pi :X \to Y$ be a quotient map between a metric space $X$  and a topological space $Y$ such that there is a measure $\mu$ on $Y$ such that for every $r_0>0$ and every $x,x' \in X$ with $\pi (x)=\pi(x')$  there is $C>0$ such that
      \begin{equation}\label{Ahlfors27ott.112}
\mu (\pi (B(x,r))) \leq C \mu (\pi (B(x',r))),
\end{equation}
for  every $r\in (0,r_0).$
   
We also assume that $\phi :Y \to X$ is an intrinsically $\eta$-quasi-symmetric section of $\pi$ such that 
\begin{enumerate}
\item $\phi (Y)$ is $Q$-Ahlfors-David regular with respect to  the measure $\phi_* \mu$, with $Q\in (0,\infty)$
\item it holds
\begin{equation}\label{equationeta}
\ell _\eta := \sup _{\substack{g,q \in \phi (Y) \\ \pi (g)=\pi(q)} }   \eta \left( \frac{d(g, \pi^{-1} (\bar y))}{d(q, \pi^{-1} (\bar y))} \right) <\infty ,
\end{equation}
for any $\bar y \in Y$ such that $g, q \notin \pi ^{-1}(\bar y)$
\end{enumerate}

  Then, for every intrinsically $\eta$-quasi-symmetric section $\psi :Y \to X,$  the set $ \psi (Y) $ is $Q$-Ahlfors-David regular with respect to  the measure $\psi_* \mu$, with $Q\in (0,\infty)$.

    \end{theorem}
 Namely, in Theorem~\ref{thm2}  $Q$-Ahlfors-David regularity means that    the measure $\phi_* \mu$ is such that  for each point $x\in \phi (Y)$ there exist $r_0>0$ and $C>0$ so that
  \begin{equation}\label{Ahlfors_IN_N}
 C^{-1}r^Q\leq  \phi_* \mu \big( B(x,r) \cap \phi (Y)\big) \leq C r^Q, \qquad \text{ for all }r\in (0,r_0).
\end{equation}


We need to a preliminary result.
 \begin{lem}\label{pseudodistance}
 Let $X$ be a metric  space, $Y$   a topological space, and $\pi:X\to Y$ a quotient map. If $\phi :Y \to X$ is an intrinsically $\eta$-quasi-symmetric section of $\pi$ such that  \eqref{equationeta} holds, then
		\begin{equation}\label{inclusionepalle}
\pi \left(B\left(p, r \right) \right) \subset \pi ( B(p, \ell_\eta r ) \cap \phi (Y)) \subset \pi (B(p, \ell _\eta r)), \quad \forall p\in \phi (Y), \forall r>0.
\end{equation}
  \end{lem}

  \begin{proof}
 Regarding the first inclusion, fix $p=\phi (y)\in \phi (Y), r>0$ and $q\in B(p, r )$ with $q\ne p .$ 


   We need to show that $\pi (q) \in \pi (\phi (Y) \cap B(p,\ell _\eta r )).$ Actually, it is enough to prove that 
\begin{equation}\label{equation2.6}
\phi (\pi (q)) \in B(p,\ell_\eta r),
\end{equation}
 because if we take $g:= \phi (\pi (q)),$ then $g\in \phi (Y)$ and 
 \begin{equation*}
 \pi (g)= \pi (\phi (\pi (q))) =\pi (q) \in \pi (\phi (Y) \cap B(p,\ell _\eta r)).
\end{equation*}
 
 Hence using the intrinsic $\eta$-quasi-symmetric property of $\phi$ and  \eqref{equationeta},  we have that for any $p=\phi (y), q, g \in \phi (Y)$ with $g= \phi (\pi (q)),$
 \begin{equation}
d(p,g) =  \frac{d(p,g)}{d(p,q)} d(p,q) \leq \eta \left( \frac{d(g, \pi ^{-1}(y))}{d(q,\pi ^{-1}(y))} \right) r \leq \ell _\eta r,
\end{equation}
i.e., \eqref{equation2.6} holds, as desired.  Finally, the second inclusion in \eqref{inclusionepalle} follows immediately noting that $\pi (\phi (Y))=Y$ because $\phi$ is a section and the proof is complete. 
  \end{proof}

  Now we are able to prove Theorem $\ref{thm2}$.
  \begin{proof}[Proof of Theorem $\ref{thm2}$]
   Let $\phi$ and $\psi$ intrinsically $\eta$-quasi-symmetric sections. 
  Fix $y\in Y.$  By Ahlfors regularity of $\phi (y),$ we know that there are $c_1,c_2, r_0>0$ such that 
     \begin{equation}\label{AhlforsNEW0}
{c_1} r^Q \leq  \phi _* \mu \big( B(\phi (y),r) \cap \phi (Y)\big)  \leq c_2 r^Q,
\end{equation}
for all $ r \in (0,r_0).$ We would like to show that there is $c_3,c_4 >0$ such that
   \begin{equation}\label{AhlforsNEW127}
 c_4 r^Q \leq \psi _* \mu \big( B(\psi(y),r) \cap \psi (Y)\big)  \leq c_4 r^Q,
\end{equation}
  for every  $ r \in (0,r_0).$ 
  We begin noticing that, by symmetry and  \eqref{Ahlfors27ott.112}
   \begin{equation}\label{Ahlfors27ott}
C^{-1} \mu (\pi (B(\psi(y),r))) \leq \mu (\pi (B(\phi(y),r))) \leq C \mu (\pi (B(\psi(y),r))).
\end{equation}
Moreover, 
\begin{equation}\label{serveperAhlfors27}
  \psi _* \mu \big( B(\psi(y),r) \cap \psi (Y)\big) =  \mu ( \psi^{-1} \big( B(\psi(y),r) \cap \psi (Y)\big) ) = \mu ( \pi \big( B(\psi(y),r) \cap \psi (Y)\big) ), 
\end{equation}
  and, consequently, 
\begin{equation*}\label{ugualecarnot} 
\begin{aligned} \psi _* \mu \big( B(\psi(y),r) \cap \psi (Y)\big) & \geq  \mu (\pi (B(\psi (y), r / \ell _\eta ))) \geq  C^{-1}   \mu (\pi (B(\phi (y), r / \ell _\eta ))) \\
& \geq  C^{-1}   \mu (\pi (B(\phi (y), r / \ell _\eta ) \cap \phi (Y)))\\
& =    C^{-1} \phi _* \mu \big( B(\phi(y), r / \ell _\eta ) \cap \phi (Y)\big) \\ & \geq   c_1C^{-1}  \ell _\eta^{-Q}  r^Q,
\end{aligned}
\end{equation*}
where in the first inequality we used  the first inclusion of \eqref{inclusionepalle}  with $\psi$ in place of $\phi$, and in the second one we used \eqref{Ahlfors27ott}. In the  third  inequality we used the second inclusion of \eqref{inclusionepalle} and in the fourth one we used  \eqref{serveperAhlfors27}  with $\phi$ in place of $\psi.$ Moreover, in a similar way we have that 
\begin{equation*}
\begin{aligned} 
\psi _* \mu \big( B(\psi(y),r) \cap \psi (Y)\big) & \leq  \mu (\pi (B(\psi (y),r))) \leq C  \mu (\pi (B(\phi (y),r)))\\
& \leq C  \mu (\pi (B(\phi (y), \ell _\eta r )) \cap \phi (Y)))
\\& = C  \phi _* \mu \big( B(\phi(y), \ell _\eta r  ) \cap \phi (Y)\big) \\
& \leq  {c_2} C  \ell _\eta^{Q} r^Q.
\end{aligned}
\end{equation*}
Hence, putting together the last two inequalities we have that \eqref{AhlforsNEW127} holds with ${c_3} =  c_1C^{-1}  \ell _\eta^{-Q}$ and $c_4 =  {c_2} C  \ell _\eta^{Q}.$ 
  \end{proof}

\subsection{Quasi-conformal sections}
In this section we present the definition of quasi-conformal sections. Regarding the classical quasi-conformal and quasi-symmetric maps the reader can see \cite{10.1007/BF02392747, A06, H01, IM01}. 

\begin{defi}\label{def_conformal}We say that a map $\phi:Y\to X$ is an {\em intrinsically $\eta$-quasi-conformal section of $\pi$}, if it is a section, i.e.,
\begin{equation}
\pi \circ \phi =\mbox{id}_Y,
\end{equation}
and there exist $H \geq 0$ and an homeomorphism $\eta :(0,\infty) \to (0,\infty)$ (i.e., $\eta$ such that for any $y_1, y_2,y_3\in Y$ distinct points of $Y$ which not belong to the same fiber, it holds
\begin{equation}\label{equation10aprile}
\frac{d(\phi (y_1), \phi (y_2))}{d(\phi (y_1), \phi (y_3))}\leq \limsup_{\substack{x,x' \in \phi (Y),\, \pi (x)=\pi(x') \\ x\to x'}  } \, \eta\left( \frac{d(x, \pi^{-1} (y_1))}{d(x', \pi^{-1} (y_1))} \right) < H.
\end{equation}
Here $d$ denotes the distance on $X$, and, as usual, for a subset $A\subset X$ and a point $x\in X$, we have
$d(x,A):=\inf\{d(x,a):a\in A\}$.
\end{defi}



Finally, this class of section satisfies the hypothesis \eqref{equationeta} of Theorem \ref{thm2}. Hence, we can conclude with the following corollary.

  \begin{theorem}[Ahlfors-David regularity]
   Let $\pi :X \to Y$ be a quotient map between a metric space $X$  and a topological space $Y$ such that there is a measure $\mu$ on $Y$ such that for every $r_0>0$ and every $x,x' \in X$ with $\pi (x)=\pi(x')$  there is $C>0$ such that
      \begin{equation}\label{Ahlfors27ott.112}
\mu (\pi (B(x,r))) \leq C \mu (\pi (B(x',r))),
\end{equation}
for  every $r\in (0,r_0).$
   
We also assume that $\phi :Y \to X$ is an intrinsically $(\eta , H)$-quasi-conformal section of $\pi$ such that  $\phi (Y)$ is $Q$-Ahlfors-David regular with respect to  the measure $\phi_* \mu$, with $Q\in (0,\infty)$ 
for some fixed $\bar y,\bar y_1 \in Y.$

  Then, for every intrinsically $(\eta , H)$-quasi-conformal section $\psi :Y \to X,$  the set $ \psi (Y) $ is $Q$-Ahlfors-David regular with respect to  the measure $\psi_* \mu$, with $Q\in (0,\infty)$.

    \end{theorem}

 {\bf Conflict of interest.}  On behalf of all authors, the corresponding author states that there is no conflict of interest.

 \bibliographystyle{alpha}
\bibliography{DDLD}

\end{document}